\documentclass[twocolumn]{jsiamletters}
\usepackage{graphicx}
\usepackage{amsmath,amssymb,amsthm}
\usepackage{bm}
\newcommand{\p}{\partial}

\newtheorem{Def}{Definition}[section]
\newtheorem{Thm}[Def]{Theorem}
\newtheorem{Lem}[Def]{Lemma}
\newtheorem{Cor}[Def]{Corollary}

\group{
Algorithms for Matrix / Eigenvalue Problems and their Applications
}
\affiliation{The University of Electro-Communications}{1-5-1 Chofugaoka, Chofu, Tokyo, 182-8585, Japan}
\authorinfo{Kurumi Shiroma}{1}{yusaku.yamamoto@uec.ac.jp}
\authorinfo{Yusaku Yamamoto}{1}{yusaku.yamamoto@uec.ac.jp}
\title{Fixed-point analysis of Ogita-Aishima's symmetric eigendecomposition refinement algorithm for multiple eigenvalues}
\abstract{Recently, Ogita and Aishima proposed an efficient eigendecomposition refinement algorithm for the symmetric eigenproblem. Their basic algorithm involves division by the difference of two approximate eigenvalues, and can become unstable when there are multiple eigenvalues. To resolve this problem, they proposed to replace those equations that casue instability with different equations and gave a convergence proof of the resulting algorithm.  However, it is not straightforward to understand intuitively why the modified algorithm works, because it removes some of the necessary and sufficient conditions for obtaining the eigendecomposition. We give an answer to this question using Banach's fixed-point theorem.}
\begin{document}

\maketitle

\section{Introduction}
Let $A\in\mathbb{R}^{n\times n}$ be a symmetric matrix with possibly multiple eigenvalues and $\tilde{X}\in\mathbb{R}^{n\times n}$ be an approximation to one of its orthogonal eigenvector matrices. We assume that $\tilde{X}$ is nonsingular. Then, there is a (non-unique) matrix $E\in\mathbb{R}^{n\times n}$ such that $X\equiv \tilde{X}(I_n+E)$, where $I_n$ is the identity matrix of order $n$, is an orthogonal eigenvector matrix of $A$. We call the problem of finding one of such $E$'s, or its approximation, the {\it symmetric eigendecomposition refinement} problem. Such a problem arises, for example, when one seeks to compute an eigendecomposition of $A$ in quadruple precision using the result computed in double precision arithmetic. Another application is solution of the time-dependent eigenvalue problem, where one seeks to compute the eigendecomposition of $A(t+\Delta t)$ using the eigendecomposition of $A(t)$, assuming that $A(t)$ and $A(t+\Delta t)$ is only slightly different.

Recently, Ogita and Aishima proposed an efficient algorithm for the symmetric eigendecomposition refinement problem that is applicable when $\tilde{X}$ is sufficiently accurate (that is, when there exists $E$ of sufficiently small norm) \cite{Ogita18}. The basic idea of their algorithm is simple. First, we consider necessary and sufficient conditions for the above $X$ to be an orthogonal eigenvector matrix of $A$: (i) $X^{\top}X=I_n$ and (ii) $D\equiv X^{\top}AX$ is diagonal. Then we rewrite these conditions in terms of $E$ and ignore second order terms in $E$. By solving the resulting linear simultaneous equations, we can easily obtain an approximation to $E$ and thus construct a better approximation to $X$. This process is repeated until the approximate orthogonal eigenvector matrix is sufficiently accurate. In \cite{Ogita18}, it is shown that the algorithm converges quadratically if $\tilde{X}$ is sufficiently accurate. Moreover, the algorithm is rich in matrix-matrix multiplications and therefore can be executed very efficiently on high performance architectures.

When $A$ has multiple eigenvalues, however, the simple algorithm described above no longer works. This is because the solution of the linear simultaneous equations for $E$ involves a division by a difference of approximate eigenvalues of $A$ (which are computed in the algorithm), and this causes division by zero or by a very small quantity when there are multiple eigenvalues. This reflects the indeterminacy of $E$ (the freedom of rotation within the eigenspace corresponding to multiple eigenvalues) in the case of multiple eigenvalues. To resolve the problem, Ogita and Aishima propose to introduce an additional constraint that determine $E$ uniquely and instead remove those equations that can cause division by zero or by a very small quantity. See subsection \ref{multiple_eigenvalues} for details. As a result, the potential instability of the algorithm is cured and the modified algorithm is shown to converge ultimately quadratically even in the presence of multiple eigenvalues. Thus, from a mathematical point of view, the problem has been completely resolved.

However, it is not straightforward to understand {\it intuitively} why the modified algorithm works. The equations removed from the modified algorithm constitute part of the necessary and sufficient conditions for $X=\tilde{X}(I_n+E)$ to be an orthogonal eigenvector matrix of $A$. Then, how can it be justified to remove them? Of course, it is justified because the resulting algorithm is proved to converge to a solution. But, it seems to be an {\it indirect} answer, although mathematically correct. A more direct explanation will facilitate deeper understanding of the algorithm. In this paper, we try to answer this question using Banach's fixed-point theorem. Note that the present paper is focused on multiple eigenvalues and not on clustered eigenvalues, which pose a more subtle problem. Consult \cite{Ogita19,Shiroma19} for recent advances on this topic.

In the following, we use the term {\it eigenvector matrix} to denote an orthogonal eigenvector matrix. $\|\cdot\|_2$ and $\|\cdot\|_F$ denote the 2-norm and the Frobenius norm of a matrix, respectively, while $\|\cdot\|_2$ is used to denote the Euclidean norm of a vector.

\section{Ogita-Aishima's algorithm}
\subsection{The basic algorithm}
In this section, we describe Ogita-Aishima's eigen-decomposition refinement algorithm, starting with its basic version. Let $A$ and $\tilde{X}$ be as defined in the Introduction. We make the following definition.
\begin{Def}
Let $\mathcal{X}$ be the set of eigenvector matrices of $A$. We define the set $\mathcal{E}_0$ of $n\times n$ matrices by
\begin{equation}
\mathcal{E}_0=\{E\in\mathbb{R}^{n\times n} \,|\,\tilde{X}(I_n+E)\in\mathcal{X}, \|E\|_2<1/3\}.
\end{equation}
\end{Def}
We assume $\mathcal{E}_0\ne\emptyset$ and seek for a matrix $E\in\mathcal{E}_0$ and the corresponding eigenvector matrix $X$. To this end, we substitute the equation $X=\tilde{X}(I+E)$ into $X^{\top}X=I_n$ and $D\equiv X^{\top}AX$, obtaining the following equations.
\begin{align}
(I_n+E)^{\top}\tilde{X}^{\top}\tilde{X}(I_n+E) &= I_n, \label{eq:OA1} \\
(I_n+E)^{\top}\tilde{X}^{\top}A\tilde{X}(I_n+E) &= D \label{eq:OA2}.
\end{align}
Note that $D$ is also an unknown matrix to be determined from (\ref{eq:OA1}) and (\ref{eq:OA2}). Apparently, its diagonal elements are the eigenvalues of $A$. After some manipulation, these equations can be rewritten as follows \cite{Ogita18}.
\begin{align}
E + E^{\top} &= R+\Delta_1(E), \label{eq:OA3} \\
D-DE-E^{\top}D &= S+\Delta_2(E,D), \label{eq:OA4}
\end{align}
where $R\equiv I_n-\tilde{X}^{\top}\tilde{X}$ and $S\equiv \tilde{X}^{\top}A\tilde{X}$ are constant matrices and $\Delta_1$ and $\Delta_2$ are matrix functions defined for an arbitrary square matrix $E\in\mathbb{R}^{n\times n}$ with $\|E\|_2<1$ and a diagonal matrix $D\in\mathbb{R}^{n\times n}$ as
\begin{align}
\Delta_1(E)  &\equiv \Delta(E)+(\Delta(E))^{\top} \nonumber \\
& \quad +(E-\Delta(E))^{\top}(E-\Delta(E)), \label{eq:Delta1definition} \\
\Delta_2(E,D)  &\equiv -D\Delta(E)-(D\Delta(E))^{\top} \nonumber \\
& \quad-(E-\Delta(E))^{\top}D(E-\Delta(E)), \label{eq:Delta2definition} \\
\Delta(E) &\equiv (I+E)^{-1}-I+E = \sum_{k=2}^{\infty}(-E)^k. \label{eq:Deltadefinition}
\end{align}
Note that $\Delta_1(E)$ and $\Delta_2(E,D)$ contain only quadratic or higher order terms in $E$.

To find $E$ and $D$ that satisfy (\ref{eq:OA3}) and (\ref{eq:OA4}), Ogita and Aishima introduce an approximation by ignoring the second order terms $\Delta_1(E)$ and $\Delta_2(E,D)$. Let us write the solution of these approximated equations as $\tilde{E}=(e_{ij})$ and $\tilde{D}={\rm diag}(\tilde{d}_1,\ldots,\tilde{d}_n)$. Then we have
\begin{align}
\tilde{e}_{ij}+\tilde{e}_{ji} &= r_{ij}, \label{eq:fij1} \\
\tilde{d}_i\delta_{ij}-\tilde{d}_i\tilde{e}_{ij}-\tilde{d}_j\tilde{e}_{ji} &= s_{ij}, \label{eq:fij2}
\end{align}
for $i,j=1,\ldots, n$. By considering the case of $i=j$ in (\ref{eq:fij1}) and (\ref{eq:fij2}), we easily get
\begin{equation}
\tilde{e}_{ii}=\frac{r_{ii}}{2}, \quad \tilde{d}_i=\frac{s_{ii}}{1-r_{ii}}.
\label{eq:ditilde}
\end{equation}
Furthermore, if $\tilde{d}_i\ne\tilde{d}_j$ for $i\ne j$, we have
\begin{equation}
\tilde{e}_{ij} = \frac{s_{ij}+\tilde{d}_j r_{ij}}{\tilde{d}_j-\tilde{d}_i} \quad (i\ne j). \label{eq:fij}
\end{equation}
Hence, once $R$ and $S$ have been computed by matrix-matrix multiplications, all the elements of $\tilde{E}$ and $\tilde{D}$ can be computed in $O(n^2)$ work. This makes Ogita-Aishima's algorithm a very efficient method for eigendecomposition refinement.

\subsection{The case of multiple eigenvalues}
\label{multiple_eigenvalues}
When $A$ has multiple eigenvalues, however, $\tilde{d}_j-\tilde{d}_i$ in the denominator of (\ref{eq:fij}) can become zero or arbitrarily small. Thus, division by zero or by a very small quantity can occur, causing breakdown or instability of the algorithm. This reflects the indeterminacy of $X$ (and therefore of $E$) in the presence of multiple eigenvalues.

To resolve the problem, Ogita and Aishima introduce an additional constraint on $E$. Their solution is based on Theorem \ref{Theorem_unique_F} below. Assume that $A$ has distinct eigenvalues $\lambda_1, \lambda_2, \ldots, \lambda_m$ with multiplicity $n_1, n_2, \ldots, n_k$, respectively. For $E\in\mathcal{E}_0$, let $X=\tilde{X}(I+E)$ and denote by $I_k$ the set of indices of the $n_k$ column vectors of $X$ that are eigenvectors belonging to $\lambda_k$. Then it can be shown that the sets $I_1, I_2, \ldots, I_m$ are determined uniquely independent of the choice of $E$ (\cite{Shiroma19}, Theorem 2.2). Now we define a superset $\mathcal{E}$ of $\mathcal{E}_0$ as a set of matrices $E^{\prime}$ such that $X^{\prime}\equiv\tilde{X}(I+E^{\prime})\in\mathcal{X}$ and the index sets $I_1^{\prime}, I_2^{\prime}, \ldots, I_m^{\prime}$ of $X^{\prime}$ are identical to $I_1, I_2, \ldots, I_m$. Then we have the following theorem \cite{Ogita18,Shiroma19}.
\begin{Thm}
\label{Theorem_unique_F}
Assume that $\mathcal{E}_0\ne\emptyset$. Then, there exists a unique matrix $F^*\in\mathbb{R}^{n\times n}$ such that
\begin{itemize}
\item[(i)] $F^*\in\mathcal{E}$,
\item[(ii)] $\|F^*\|_2 < 3\min_{E\in\mathcal{E}_0}\|E\|_2 < 1$,
\item[(iii)] The $m$ principal submatrices $F_{I_1 I_1}^*, F_{I_2 I_2}^*, \ldots, F_{I_m I_m}^*$ of $F^*$ are all symmetric,
\end{itemize}
where $F_{I_k I_k}^*$ is an $n_k\times n_k$ principal submatrix of $F^*$ consisting of rows and columns with indices in $I_k$.
\end{Thm}
Hence, if we choose to seek for this $F^*$, we can avoid the arbitrarity arising from multiple eigenvalues.

Based on this idea, Ogita and Aishima propose to use the symmetry condition of (iii) {\it instead of} part of the condition (\ref{eq:fij2}). To be concrete, let $\mathcal{K}(i)$ be the index of the eigenvalue ($1\le \mathcal{K}(i) \le m$) to which the $i$th column vector of $X^*$ belongs. Ogita and Aishima propose to replace (\ref{eq:fij1}) and (\ref{eq:fij2}) with the following equations.
\begin{align}
\tilde{f}_{ij}+\tilde{f}_{ji} &= r_{ij} \quad (i,j=1, \ldots, n), \label{eq:fij3} \\
\tilde{d}_i\delta_{ij}-\tilde{d}_i\tilde{f}_{ij}-\tilde{d}_j\tilde{f}_{ji} &= s_{ij} \quad (\mathcal{K}(i)\ne\mathcal{K}(j) {\rm \;or\;} i=j), \label{eq:fij4} \\
\tilde{f}_{ij} &= \tilde{f}_{ji} \quad (\mathcal{K}(i)=\mathcal{K}(j)). \label{eq:fij5}
\end{align}
Then, instead of (\ref{eq:fij}), we have
\begin{align}
\tilde{f}_{ij} &= \frac{s_{ij}+\tilde{d}_j r_{ij}}{\tilde{d}_j-\tilde{d}_i} \quad (\mathcal{K}(i)\ne\mathcal{K}(j)), \label{eq:fij6} \\
\tilde{f}_{ij} &= \frac{r_{ij}}{2} \quad (\mathcal{K}(i)=\mathcal{K}(j)). \label{eq:fij7}
\end{align}
In the actual algorithm, since it is difficult to judge from computed quantities whether $\mathcal{K}(i)=\mathcal{K}(j)$ or not, they instead switch between (\ref{eq:fij6}) and (\ref{eq:fij7}) depending on whether $|\tilde{d}_j-\tilde{d}_i|<\delta_1$ or not, for some threshold $\delta_1$. Thus, division by a quantity smaller than $\delta_1$ is avoided.


\subsection{The question treated in this paper}
While this solution looks nice, it raises a natural question. In the modified algorithm, (\ref{eq:fij2}) for the case of $\mathcal{K}(i)=\mathcal{K}(j)$ and $i\ne j$ is {\it ignored} and (\ref{eq:fij5}) is used instead. However, the former is part of the necessary and sufficient condition for $X^*=\tilde{X}(I+F^*)$ to be an eigenvector matrix of $A$. Then, how can it be justified to ignore it?

Ogita and Aishima answer this question by proving that the modified algorithm delivers correct $F^*$ even though it ignores (\ref{eq:fij2}) partly. From a mathematical point of view, this gives a complete answer. However, from the proof, it is difficult to understand intuitively why this omission is justified. In the next section, we seek to give a more direct answer to this question.

\section{Fixed-point analysis of Ogita-Aishima's algorithm}
\subsection{The idea}
To answer the question, we go back to the {\it exact} equations (with quadratic and higher order terms) that define $F^*$ and $D^*$, which can be written as follows.
\begin{align}
f_{ij}^*+f_{ji}^* &= r_{ij}+(\Delta_1(F^*))_{ij} \nonumber \\
& \quad\quad (i,j=1, \ldots, n), \label{eq:fij8} \\
d_i^*\delta_{ij}-d_i^* f_{ij}^*-d_j^*f_{ji}^* &= s_{ij}+(\Delta_2(F^*,D^*))_{ij} \nonumber \\
& \quad\quad(i,j=1, \ldots, n), \label{eq:fij9} \\
f_{ij}^* &= f_{ji}^* \quad (\mathcal{K}(i)=\mathcal{K}(j)). \label{eq:fij10}
\end{align}
Note that the index pair $(i,j)$ runs over {\it all} possible pairs in (\ref{eq:fij9}). Now, we restrict the range of $(i,j)$ in (\ref{eq:fij9}), as in Ogita-Aishima's modified algorithm.
\begin{align}
d_i^*\delta_{ij}-d_i^* f_{ij}^*-d_j^* f_{ji}^* &= s_{ij}+(\Delta_2(F^*,D^*))_{ij} \nonumber \\
& (\mathcal{K}(i)\ne\mathcal{K}(j) {\rm \;or\;} i=j), \label{eq:fij11}
\end{align}
In the following, we try to justify solving (\ref{eq:fij8}), (\ref{eq:fij11}) and (\ref{eq:fij10}) (which we call {\bf problem P'}) instead of (\ref{eq:fij8}), (\ref{eq:fij9}) and (\ref{eq:fij10}) (which we call {\bf problem P}).

Our strategy is as follows. For convenience, let us express the pair $(F,D)$ by an $(n^2+n)$-dimensional vector ${\bm z}$. We denote the vector corresponding to $(F^*,D^*)$ by ${\bm z}^*$ and an closed sphere in $\mathbb{R}^{n^2+n}$ with center ${\bm z}^*$ and radius $\delta$ by $B_{\delta}({\bm z}^*)$. Now, suppose that there exists some $\delta>0$ such that the problem {\bf P'} has a unique solution ${\bm z}^{\prime *}$ in $B_{\delta}({\bm z}^*)$. Then, ${\bm z}^{\prime *}={\bm z}^*$, since otherwise {\bf P'} would have two solutions in $B_{\delta}({\bm z}^*)$ (Note that the solution of {\bf P} is also a solution of {\bf P'}). This means that if we solve {\bf P'} near the true solution of {\bf P}, we exactly get the true solution ${\bm z}^*$ of {\bf P}, justifying solving {\bf P'} instead of {\bf P}. Thus the problem is reduced to finding such $\delta$.

\subsection{Construction of a contraction mapping on $B_{\delta}({\bm z}^*)$}
To show the existence of $\delta>0$ defined in the previous subsection, we resort to Banach's fixed-point theorem. To this end, we first rewrite (\ref{eq:fij8}), (\ref{eq:fij11}) and (\ref{eq:fij10}) as a mapping from $\mathbb{R}^{n^2+n}$ to itself. Let $\eta\equiv\min_{k\ne \ell}|\lambda_k-\lambda_{\ell}|$ and assume that $\delta<\frac{\eta}{2}$. Then, if ${\bm z}\in B_{\delta}({\bm z^*})$, for the pair $(F, D)$ corresponding to ${\bm z}$, $|d_j-d_i|\ge\eta-2\delta>0$ when $\mathcal{K}(i)\ne\mathcal{K}(j)$. Thus, we can rewrite (\ref{eq:fij8}), (\ref{eq:fij11}) and (\ref{eq:fij10}) as the following mapping from ${\bm z}$ to $\hat{\bm {z}}$ (see the derivation of (\ref{eq:ditilde}), (\ref{eq:fij6}) and (\ref{eq:fij7})).
\begin{align}
\hat{d}_i &= \frac{s_{ii}+(\Delta_2(F,D))_{ii}}{1-r_{ii}-(\Delta_1(F))_{ii}} \quad (i=1, \ldots, n), \label{eq:hatdi} \\
\hat{f}_{ij} &= \frac{s_{ij}+(\Delta_2(F,D))_{ij}+d_j (r_{ij}+(\Delta_1(F))_{ij})}{d_j-d_i} \nonumber \\
& \quad\quad\quad\quad\quad\quad\quad\quad\quad\; (\mathcal{K}(i)\ne\mathcal{K}(j)), \label{eq:hatfij1} \\
\hat{f}_{ij} &= \frac{r_{ij}+(\Delta_1(F))_{ij}}{2} \quad (\mathcal{K}(i)=\mathcal{K}(j)). \label{eq:hatfij2}
\end{align}
Let us write this mapping as $\hat{\bm z}={\mathcal F}({\bm z})$. Then, the fixed point of $\mathcal{F}$ corresponds to the solution of ${\bf P'}$. According to Banach's fixed-point theorem, if $\mathcal{F}$ is a contraction mapping on $B_{\delta}({\bf z^*})$, it has a unique fixed point in $B_{\delta}({\bf z^*})$.

Now we prove that $\mathcal{F}$ is a contraction mapping from $B_{\delta}({\bm z}^*)$ to itself for some $\delta>0$ if $\|F^*\|_2$ is sufficiently small (that is, if $\tilde{X}$ is a sufficiently good approximation to one of the eigenvector matrix of $A$). For this, we need to show that $\mathcal{F}$ has the following two properties.
\begin{itemize}
\item[(i)] $\mathcal{F}(B_{\delta}({\bm z}^*)) \subseteq B_{\delta}({\bm z}^*)$.
\item[(ii)] There exists a constant $c<1$ such that for any ${\bm z}_1,{\bm z}_2 \in B_{\delta}({\bm z}^*)$, $\|\mathcal{F}({\bm z}_1)-\mathcal{F}({\bm z}_2)\| \le c\|{\bm z}_1-{\bm z}_2\|$.
\end{itemize}
Here, $\|\cdot\|$ denotes the Euclidean norm in $\mathbb{R}^{n^2+n}$, which is equal to $\sqrt{\|F\|_F^2+\|D\|_F^2}$. Note that (i) can be derived from (ii). In fact, suppose that (ii) holds and ${\bm z}\in B_{\delta}({\bm z}^*)$. Then, $\|{\bm z}-{\bm z}^*\|\le\delta$ from the definition and
\begin{align}
\|\mathcal{F}({\bm z})-{\bm z}^*\| &= \|\mathcal{F}({\bm z})-\mathcal{F}({\bm z}^*)\| \nonumber \\
&\le c\|{\bm z}-{\bm z}^*\| \le c\delta < \delta,
\end{align}
which shows $\mathcal{F}({\bm z})\in B_{\delta}({\bm z}^*)$ and therefore $\mathcal{F}(B_{\delta}({\bm z}^*)) \subseteq B_{\delta}({\bm z}^*)$. Here, we used the fact that ${\bm z}^*$ is a solution of {\bf P} and hence also a solution of {\bf P'} and a fixed point of $\mathcal{F}$. Thus, the remaining task is to show (ii).

To show (ii), we use the following mean value theorem in the Banach space \cite{Cook94}.
\begin{Thm}
Let $Z$ and $W$ be Banach spaces and $\mathcal{G}$ be a G\^{a}teaux differentiable mapping from an open set $U\subseteq Z$ to $W$. Let $[{\bm z}_1, {\bm z}_2]$ denote the line segment joining two points ${\bm z}_1, {\bm z}_2$ in $U$. Then
\begin{equation}
\|\mathcal{G}({\bm z}_1)- \mathcal{G}({\bm z}_2)\| \le \|{\bm z}_1-{\bm z}_2\|\sup_{{\bm z}\in[{\bm z}_1, {\bm z}_2]}\|D\mathcal{G}({\bm z})\|,
\end{equation}
where $D\mathcal{G}({\bm z})$ is the G\^{a}teaux derivative of $\mathcal{G}$.
\end{Thm}
We choose $U$ to be an open sphere obtained by expanding the radius of $B_{\delta}({\bm z}^*)$ by $\epsilon\equiv\frac{1}{2}(\frac{\eta}{2}-\delta)$, so that $\mathcal{F}$ is defined also on $U$. Also, if ${\bm z}_1, {\bm z}_2 \in B_{\delta}({\bm z}^*)$, we can replace $\sup_{{\bm z}\in[{\bm z}_1, {\bm z}_2]}$ in the right-hand side with $\sup_{{\bm z}\in B_{\delta}({\bm z}^*)}$. Moreover, our $\mathcal{F}$ is a mapping between finite dimensional Banach spaces and each element of $\mathcal{F}$, as defined by (\ref{eq:hatdi}), (\ref{eq:hatfij1}) and (\ref{eq:hatfij2})), is a rational function of $\{f_{ij}\}_{i,j=1}^n$ and $\{d_i\}_{i=1}^n$ when $\|F\|_2<1$, as can be easily seen from (\ref {eq:Delta1definition}), (\ref {eq:Delta2definition}) and (\ref {eq:Deltadefinition}). Hence, we can use the Jacobian matrix $\frac{\partial\mathcal{F}}{\partial{\bm z}}$ instead of the G\^{a}teaux derivative. Thus, we obtain the following corollary.
\begin{Cor}
For any ${\bm z}_1,{\bm z}_2 \in B_{\delta}({\bm z}^*)$,
\begin{equation}
\|\mathcal{F}({\bm z}_1)-\mathcal{F}({\bm z}_2)\| \le \|{\bm z}_1-{\bm z}_2\|\sup_{{\bm z}\in B_{\delta}({\bm z}^*)}\left\|\frac{\partial\mathcal{F}}{\partial{\bm z}}\right\|_2.
\end{equation}
\end{Cor}
Thus, all we need to do is to show that $\left\|\frac{\partial\mathcal{F}}{\partial{\bm z}}\right\|_2\le c<1$ when ${\bm z}\in B_{\delta}({\bm z}^*)$. Since $\frac{\partial\mathcal{F}}{\partial{\bm z}}$ is an $(n^2+n)\times(n^2+n)$ matrix and $\left\|\frac{\partial\mathcal{F}}{\partial{\bm z}}\right\|_2 \le \left\|\frac{\partial\mathcal{F}}{\partial{\bm z}}\right\|_F$, it is sufficient to make sure that the modulus of each element of $\frac{\partial\mathcal{F}}{\partial{\bm z}}$ is bounded by $\frac{c}{n^2+n}$, where $c<1$. Note that each element of $\frac{\partial\mathcal{F}}{\partial{\bm z}}$ has a form of $\frac{\p \hat{f}_{ij}}{\p f_{k\ell}}$, $\frac{\p\hat{f}_{ij}}{\p d_k}$, $\frac{\p \hat{d}_i}{\p f_{k\ell}}$ or $\frac{\p\hat{d}_i}{\p d_k}$. Since the functional form of $\hat{f}_{ij}$ is different depending on whether $\mathcal{K}(i)=\mathcal{K}(j)$ or not, we need to evaluate six kinds of elements.

As a preparation, we evaluate intermediate quantities appearing in the evaluation of these partial derivatives.
\begin{Lem}
When $\|F\|_2<\frac{1}{10}$ and $\|F^*\|_2<\frac{1}{10}$, the following inequalities hold. Here, $1\le i,j,k,\ell\le n$ unless otherwise noted.
\begin{align}
\|\Delta_1(F)\| &\le \frac{1}{2}\|F\|_2 \le \frac{1}{20}, \label{eq:Delta1bound} \\
\|\Delta_2(F,D)\| &\le \frac{1}{2}\|D\|_2\|F\|_2 \le \frac{1}{20}\|D\|_2, \label{eq:Delta2bound} \\
\frac{\p}{\p f_{k\ell}}(\Delta_1(F))_{ij} &\le 9\|F\|_2, \label{eq:dDelta1F} \\
\frac{\p}{\p f_{k\ell}}(\Delta_2(F,D))_{ij} &\le 9\|D\|_2\|F\|_2, \label{eq:dDelta2F} \\
\frac{\p}{\p d_k}(\Delta_2(F,D))_{ij} &\le 4\|F\|_2^2, \label{eq:dDelta2F2} \\
|r_{ij}| &\le \frac{5}{2}\|F^*\|_2, \label{eq:rijbound} \\
|s_{ij}| &\le \begin{cases}
    \frac{5}{2}\|A\|_2\|F^*\|_2 & (i\ne j) \\
    \frac{5}{4}\|A\|_2 & (i=j)
  \end{cases} \label{eq:sijbound}
\end{align}
\end{Lem}

\begin{proof}
From equations (26) and (27) in \cite{Ogita18}, we have
\begin{align}
\|\Delta_1(F)\| &\le \frac{(3-2\|F\|_2)\|F\|_2^2}{(1-\|F\|_2)^2} \nonumber \\
&\le \frac{3\cdot\frac{1}{10}\|F\|_2}{(1-\frac{1}{10})^2} \nonumber \\
&\le \frac{1}{2}\|F\|_2 \le \frac{1}{20}.
\end{align}
Similarly, from equations (26) and (28) in \cite{Ogita18}, it follows that
\begin{align}
\|\Delta_2(F,D)\| &\le \frac{(3-2\|F\|_2)\|F\|_2^2}{(1-\|F\|_2)^2}\,\|D\|_2 \nonumber \\
&\le \frac{1}{2}\|D\|_2\|F\|_2 \le \frac{1}{20}\|D\|_2.
\end{align}

To derive (\ref{eq:dDelta1F}), (\ref{eq:dDelta2F}) and (\ref{eq:dDelta2F2}), we first evaluate the change of $\Delta(F)$ when $F$ changes to $F+dF$. From the definition of $\Delta(F)$ in (\ref{eq:Deltadefinition}), we have
\begin{align}
& d\Delta(F) \nonumber \\
&\equiv \Delta(F+dF)-\Delta(F) \nonumber \\
&= \sum_{k=2}^{\infty} (-1)^k\{(dF)F^{k-1}+F(dF)F^{k-2}+\cdots F^{k-1}(dF)\} \nonumber \\
& \quad\quad +O((dF)^2) \nonumber \\
&= \sum_{k=2}^{\infty} (-1)^k\sum_{\ell=0}^{k-1}F^{\ell}(dF)F^{k-1-\ell} +O((dF)^2).
\end{align}
Hence,
\begin{align}
\|d\Delta(F)\|_2 &\le \sum_{k=2}^{\infty}\sum_{\ell=0}^{k-1}\|F\|_2^{k-1}\|dF\|_2\|F\|_2^{k-1-\ell} + O(\|dF\|_2^2) \nonumber \\
&= \sum_{k=2}^{\infty}k\|F\|_2^{k-1}\|dF\|_2 + O(\|dF\|_2^2) \nonumber \\
&= \|dF\|_2\cdot\left\{\frac{1}{(1-\|F\|_2)^2}-1\right\} + O(\|dF\|_2^2) \nonumber \\
&= \|dF\|_2\cdot\frac{\|F\|_2(2-\|F\|_2)}{(1-\|F\|_2)^2} + O(\|dF\|_2^2) \nonumber \\
&\le \|dF\|_2\cdot\frac{2\|F\|_2}{(1-\frac{1}{10})^2} + O(\|dF\|_2^2) \nonumber \\
&\le 3\|F\|_2\|dF\|_2 + O(\|dF\|_2^2).
\end{align}
Using this result, we bound the left-hand side of (\ref{eq:dDelta1F}). The change of $\Delta_1(F)$ can be written as
\begin{align}
d\Delta_1(F) &= d\Delta(F) + (d\Delta(F))^{\top} \nonumber \\
& \quad\quad + (dF-d\Delta(F))^{\top}(F-\Delta(F)) \nonumber \\
& \quad\quad + (F-\Delta(F))^{\top}(dF-d\Delta(F)) \nonumber \\
& \quad\quad + O((dF)^2).
\end{align}
Hence,
\begin{align}
& \|d\Delta_1(F)\|_2 \nonumber \\
&\le 2\|d\Delta(F)\|_2 \nonumber \\
& \quad\quad + 2\|F-\Delta(F)\|_2(\|dF\|_2+\|d\Delta(F)\|_2) \nonumber \\
& \quad\quad + O(\|dF\|_2^2) \nonumber \\
&\le 6\|F\|_2\|dF\|_2 + 2\cdot\frac{10}{9}\|F\|_2(1+3\|F\|_2)\|dF\|_2 \nonumber \\
& \quad\quad + O(\|dF\|_2^2) \nonumber \\
&\le 9\|F\|_2\|dF\|_2 + O(\|dF\|_2^2),
\end{align}
where we used in the second inequality, 
\begin{align}
\|F-\Delta(F)\|_2 &= \left\|\sum_{k=1}^{\infty}(-1)^k F^k\right\| \le \sum_{k=1}^{\infty}\|F\|_2^k \nonumber \\
&= \frac{\|F\|_2}{1-\|F\|_2} \le \frac{10}{9}\|F\|_2.
\end{align}
Now we consider the case where the $(k,\ell)$ element of $dF$ is equal to $df_{k\ell}$ and all the other elements are zero. Then,
\begin{align}
|d(\Delta_1(F))_{ij}| &\le \|d\Delta_1(F)\|_2 \nonumber \\
&\le 9\|F\|_2\|dF\|_2 + O(\|dF\|_2^2) \nonumber \\
&= 9\|F\|_2|df_{k\ell}| + O(|df_{k\ell}|^2).
\end{align}
By dividing both sides by $df_{k\ell}$ and taking the limit of $df_{k\ell}\rightarrow 0$, we have (\ref{eq:dDelta1F}). Inequality (\ref{eq:dDelta2F}) can be obtained in a similar way.

Since $\Delta_2(F,D)$ is linear in $D$, when $D$ is changed to $D+dD$, it changes as
\begin{align}
d\Delta_2(F,D) &= -(dD)\Delta(F)-(\Delta(F))^{\top}(dD) \nonumber \\
& \quad\quad -(F-\Delta(F))^{\top}(dD)(F-\Delta(F)).
\end{align}
Thus,
\begin{align}
\|d\Delta_2(F,D)\|_2 &\le 2\|\Delta(F)\|_2\|dD\|_2 \nonumber \\
& \quad\quad +\|F-\Delta(F)\|_2^2\|dD\|_2 \nonumber \\
&\le \left\{2\cdot\frac{10}{9}+\left(\frac{10}{9}\right)^2\right\}\|F\|_2^2\|dD\|_2 \nonumber \\
&< 4\|F\|_2^2\|dD\|_2.
\end{align}
where we used in the second inequality,
\begin{align}
\|\Delta(F)\|_2 &= \left\|\sum_{k=2}^{\infty}(-1)^k F^k\right\| \le \sum_{k=2}^{\infty}\|F\|_2^k \nonumber \\
&= \frac{\|F\|_2^2}{1-\|F\|_2} \le \frac{10}{9}\|F\|_2^2.
\end{align}
Considering the case where the $k$th diagonal element of $dD$ is $dd_k$ and all the other elements are zero leads to
\begin{align}
|d(\Delta_2(F,D))_{ij}| &\le \|d\Delta_2(F,D)\|_2 \nonumber \\
&\le 4\|F\|_2^2\|dD\|_2 \nonumber \\
&= 4\|F\|_2^2|dd_k|.
\end{align}
Dividing both sides by $dd_k$ and taking the limit of $dd_k\rightarrow 0$ gives (\ref{eq:dDelta2F2}).

To prove (\ref{eq:rijbound}), we use the following equation satisfied by $F^*$ (see (\ref{eq:OA3})).
\begin{equation}
F^*+(F^*)^{\top}=R+\Delta_1(F^*).
\end{equation}
Using (\ref{eq:Delta1bound}), we have
\begin{align}
\|R\|_2 &\le 2\|F^*\|_2+\|\Delta_1(F^*)\|_2 \nonumber \\
&\le 2\|F^*\|_2|+\frac{1}{2}\|F^*\|_2=\frac{5}{2}\|F^*\|_2,
\end{align}
from which (\ref{eq:rijbound}) follows immediately.

Finally, we derive (\ref{eq:sijbound}). First, note that $(F^*,D^*)$ satisfies the following equation (see (\ref{eq:OA4})).
\begin{equation}
D^*-D^*F^*-(D^*F^*)^{\top}=S+\Delta_2(F^*,D^*).
\end{equation}
Thus,
\begin{align}
\|S-D^*\|_2 &\le 2\|D^*\|_2\|F^*\|_2 + \|\Delta_2(F^*,D^*)\|_2 \nonumber \\
& \le 2\|D^*\|_2\|F^*\|_2 + \frac{1}{2}\|D^*\|_2\|F^*\|_2 \nonumber \\
&= \frac{5}{2}\|D^*\|_2\|F^*\|_2 = \frac{5}{2}\|A\|_2\|F^*\|_2,
\end{align}
where we used $\|D^*\|_2=\|A\|_2$ as the diagonal elements of $D^*$ are the eigenvalues of $A$.
Since $s_{ij}$ ($i\ne j$) is an off-diagonal element of $S-D^*$, the case of $i\ne j$ in (\ref{eq:sijbound}) follows. Furthermore, from $\|F^*\|_2<\frac{1}{10}$, we have
\begin{align}
\|S\|_2 &\le \|D^*\|_2 + 2\|D^*\|_2\|F^*\|_2 + \|\Delta_2(F^*,D^*)\|_2 \nonumber \\
&\le \|D^*\|_2\left(1+\frac{5}{2}\|F^*\|_2\right) \nonumber \\
&\le \frac{5}{4}\|D^*\|_2 = \frac{5}{4}\|A\|_2,
\end{align}
from which the case of $i=j$ in (\ref{eq:sijbound}) follows.
\end{proof}

When ${\bm z}\in B_{\delta}({\bm z}^*)$ and $\eta-2\delta>0$, the bounds in (\ref{eq:Delta1bound}) through (\ref{eq:dDelta2F2}) can be rewritten in terms of $\|F^*\|_2$ and $\|A\|_2$ using the following relations.
\begin{align}
\|F\|_2 &\le \|F^*\|_2 + \|F-F^*\|_2 \nonumber \\
&\le \|F^*\|_2 + \|F-F^*\|_F \nonumber \\
&< \|F^*\|_2 + \delta, \label{eq:Frelation} \\
\|D\|_2 &\le \|D^*\|_2 + \|D-D^*\|_2 \nonumber \\
&\le \|D^*\|_2 + \|D-D^*\|_F \nonumber \\
&< \|D^*\|_2 + \delta \le 2\|A\|_2, \label{eq:Drelation}
\end{align}
where in the last inequality, we used the relation
\begin{align}
\delta \le \frac{\eta}{2} &\le \frac{1}{2}\{\lambda_{\max}(A)-\lambda_{\min}(A)\} \nonumber \\
&\le \frac{1}{2}(\|A\|_2+\|A\|_2) = \|A\|_2.
\end{align}
Here, $\lambda_{\max}(A)$ and $\lambda_{\min}(A)$ denote the largest and smallest eigenvalues, respectively, of $A$.

Using these results, it is straightforward to evaluate the elements of $\frac{\p\mathcal{F}}{\p{\bm z}}$. We have the following Lemma.
\begin{Lem}
\label{Lemma_bound_derivatives}
Suppose that $\|F^*\|<\frac{1}{20}$ and $\delta<\min\{\frac{\eta}{3},\|F^*\|_2\}$. Furthermore, let ${\bm z}\in B_{\delta}({\bm z}^*)$. Then, we have the following bounds on the modulus of each element of $\frac{\p\mathcal{F}}{\p{\bm z}}$. Here, $1\le i,j,k,\ell\le n$ unless otherwise noted.
\begin{align}
\left|\frac{\p\hat{f}_{ij}}{\p f_{k\ell}}\right|
&\le 9\|F^*\|_2 \quad (\mathcal{K}(i)=\mathcal{K}(j)), \label{eq:ffbound1} \\
\left|\frac{\p\hat{f}_{ij}}{\p d_k}\right| &= 0 \quad (\mathcal{K}(i)=\mathcal{K}(j)), \label{eq:fdbound1} \\
\left|\frac{\p\hat{f}_{ij}}{\p f_{k\ell}}\right| &\le \frac{216}{\eta}\|A\|_2\|F^*\|_2 \quad (\mathcal{K}(i)\ne\mathcal{K}(j)), \label{eq:ffbound2} \\
\left|\frac{\p\hat{f}_{ij}}{\p d_k}\right| &\le \frac{3}{\eta}\left(\frac{43}{10}+\frac{69}{2}\cdot\frac{\|A\|_2}{\eta}\right)\|F^*\|_2 \nonumber \\
& \quad\quad\quad\quad\quad\quad (\mathcal{K}(i)\ne\mathcal{K}(j)), \label{eq:fdbound2} \\
\left|\frac{\p\hat{d}_i}{\p f_{k\ell}}\right| &\le \frac{9600}{121}\|A\|_2\|F^*\|_2, \label{eq:dfbound} \\
\left|\frac{\p\hat{d}_i}{\p d_k}\right| &\le \frac{32}{33}\|F^*\|_2. \label{eq:ddbound}
\end{align}
\end{Lem}

\begin{proof}
Note that $\|F\|_2 < \|F^*\|_2+\delta < 2\|F^*\|_2 < \frac{1}{10}$ from the assumption. First, consider the case of $\mathcal{K}(i)=\mathcal{K}(j)$. We have from (\ref{eq:hatfij2}) and (\ref{eq:dDelta1F}),
\begin{align}
\left|\frac{\p\hat{f}_{ij}}{\p f_{k\ell}}\right|
&\le \frac{1}{2}\left|\frac{\p}{\p f_{k\ell}}(\Delta_1(F))_{ij}\right| \nonumber \\
&\le \frac{1}{2}\cdot 9\|F\|_2 \le 9\|F^*\|_2.
\end{align}
Moreover, (\ref{eq:fdbound1}) holds naturally since $\hat{f}_{ij}$ does not depend on $\{d_k\}_{k=1}^n$ in this case.

Second, consider the case of $\mathcal{K}(i)\ne\mathcal{K}(j)$. Differentiating (\ref{eq:hatfij1}) with respect to $f_{k\ell}$ and using (\ref{eq:dDelta1F}) and (\ref{eq:dDelta2F}) gives
\begin{align}
\left|\frac{\p\hat{f}_{ij}}{\p f_{k\ell}}\right|
&\le \frac{1}{|d_j-d_i|}\left\{\left|\frac{\p}{\p f_{k\ell}}(\Delta_2(F,D))_{ij}\right|\right. \nonumber \\
& \quad\quad +\left.\left|d_j\frac{\p}{\p f_{k\ell}}(\Delta_1(F))_{ij}\right|\right\} \nonumber \\
&\le \frac{1}{\eta-2\delta}\left(9\|D\|_2\|F\|_2+\|D\|_2\cdot9\|F\|_2\right) \nonumber \\
&\le \frac{216}{\eta}\|A\|_2\|F^*\|_2,
\end{align}
where we used $\eta-2\delta > \eta-\frac{2}{3}\eta = \frac{1}{3}\eta$ in the last inequality. 
On the other hand, Differentiating (\ref{eq:hatfij1}) with respect to $d_k$ gives
\begin{align}
\left|\frac{\p\hat{f}_{ij}}{\p d_k}\right|
&\le \frac{1}{|d_j-d_i|}\left\{\left|\frac{\p}{\p d_k}(\Delta_2(F,D))_{ij}\right|\right. \nonumber \\
& \quad\quad +\left.\left|\frac{\p}{\p d_k}\left(d_j(r_{ij}+(\Delta_1(F))_{ij})\right)\right|\right\} \nonumber \\
& \quad + \frac{1}{|d_j-d_i|^2}\left|\frac{\p}{\p d_k}(d_j-d_i)\right| \nonumber \\
& \quad\quad \times \left\{|s_{ij}| + |(\Delta_2(F,D))_{ij}|\right. \nonumber \\
& \quad\quad\quad +\left.|d_j(r_{ij}+(\Delta_1(F))_{ij})|\right\} \nonumber \\
&\le \frac{1}{\eta-2\delta}\left\{4\|F\|_2^2+\delta_{jk}\left(\frac{5}{2}\|F^*\|_2+\frac{1}{2}\|F\|_2\right)\right\} \nonumber \\
& \quad+ \frac{1}{(\eta-2\delta)^2}(\delta_{jk}+\delta_{ik}) \nonumber \\
& \quad\quad \times\left\{\frac{5}{2}\|A\|_2\|F^*\|_2+\frac{1}{2}\|D\|_2\|F\|_2\right.\nonumber \\ 
& \quad\quad\quad + \left.\|D\|_2\left(\frac{5}{2}\|F^*\|_2+\frac{1}{2}\|F\|_2\right)\right\} \nonumber \\
&\le \frac{3}{\eta}\left(\frac{4}{5}\|F^*\|_2+\frac{5}{2}\|F^*\|_2+\|F^*\|_2\right) \nonumber \\
& \quad+ \left(\frac{3}{\eta}\right)^2\left\{\frac{5}{2}\|A\|_2\|F^*\|_2+2\|A\|_2\|F^*\|_2\right. \nonumber \\
& \quad\quad +\left.2\|A\|_2\left(\frac{5}{2}\|F^*\|_2+\|F^*\|_2\right)\right\} \nonumber \\
&\le \frac{3}{\eta}\left(\frac{43}{10}+\frac{69}{2}\cdot\frac{\|A\|_2}{\eta}\right)\|F^*\|_2.
\end{align}
Now we consider the partial derivative of $\hat{d}_i$.
\begin{align}
\left|\frac{\p\hat{d}_i}{\p f_{k\ell}}\right|
&\le \frac{1}{|1-r_{ii}-(\Delta_1(F))_{ii}|}\left|\frac{\p}{\p f_{k\ell}}(\Delta_2(F,D))_{ii}\right| \nonumber \\
& \quad + \frac{1}{|1-r_{ii}-(\Delta_1(F))_{ii}|^2} \nonumber \\
& \quad\quad \times \left|\frac{\p}{\p f_{k\ell}}(\Delta_1(F))_{ii}\left\{s_{ii}+(\Delta_2(F,D))_{ii}\right\}\right| \nonumber \\
&\le \frac{1}{1-\frac{5}{2}\cdot\frac{1}{20}-\frac{1}{20}}\cdot 9\|D\|_2\|F\|_2 \nonumber \\
& \quad + \frac{1}{\left(1-\frac{5}{2}\cdot\frac{1}{20}-\frac{1}{20}\right)^2} \nonumber \\
& \quad\quad \times 9\|F\|_2\left(\frac{5}{4}\|A\|_2+\frac{1}{2}\|D\|_2\|F\|_2\right) \nonumber \\
&\le \frac{40}{33}\cdot 9\cdot 4\|A\|_2\|F^*\|_2 \nonumber \\
& \quad + \left(\frac{40}{33}\right)^2\cdot 18\|F^*\|_2\left(\frac{5}{4}\|A\|_2+\|A\|_2\cdot\frac{1}{10}\right) \nonumber \\
&= \frac{9600}{121}\|A\|_2\|F^*\|_2.
\end{align}
Finally,
\begin{align}
\left|\frac{\p\hat{d}_i}{\p d_k}\right|
&\le \frac{1}{|1-r_{ii}-(\Delta_1(F))_{ii}|}\left|\frac{\p}{\p d_k}(\Delta_2(F,D))_{ii}\right| \nonumber \\
&\le \frac{1}{1-\frac{5}{2}\cdot\frac{1}{20}-\frac{1}{20}}\cdot 4\|F\|_2^2 \nonumber \\
&\le \frac{40}{33}\cdot 4\cdot 2\|F^*\|_2\cdot\frac{1}{10} = \frac{32}{33}\|F^*\|_2.
\end{align}
\end{proof}

Since the absolute values of the derivatives in Lemma \ref {Lemma_bound_derivatives} can be made arbitrarily small by decreasing $\|F^*\|_2$, they can be made simultaneously smaller than, say, $\frac{1}{2}\frac{1}{n^2+n}$. Then, $\left\|\frac{\partial\mathcal{F}}{\partial{\bm z}}\right\|_2 \le \left\|\frac{\partial\mathcal{F}}{\partial{\bm z}}\right\|_F < \frac{1}{2}$ and $\mathcal{F}$ becomes a contraction mapping on $B_{\delta}({\bm z}^*)$. Hence, {\bf P'} has a unique solution in $B_{\delta}({\bm z}^*)$, which is also a solution to {\bf P}. Thus, we arrive at the main theorem of this paper.
\begin{Thm}
Suppose that the 2-norm of $F^*$ defined in Lemma \ref{Theorem_unique_F} is sufficiently small (that is, $\tilde{X}$ is sufficiently close to an eigenvector matrix of $A$) and choose $\delta$ so that $\delta<\min\{\frac{\eta}{3},\|F^*\|_2\}$. Then, if we solve problem {\bf P'} in an open ball $B_{\delta}({\bm z}^*)$, we obtain the solution of problem {\bf P}.
\end{Thm}
This gives a justification for solving {\bf P'} instead of {\bf P}. At the same time, this gives a rationale for Ogita-Aishima's modified algorithm for multiple eigenvalues, which solves the problem {\bf P'} approximately by ignoring the second order terms $\Delta_1(F^*)$ and $\Delta_2(F,D)$ in (\ref{eq:fij8}) and (\ref{eq:fij11}), respectively.

\section*{Acknowledgment}
This study is supported by JSPS KAKENHI Grant Numbers 17H02828, 17K19966 and 19KK02555.

\references

\end{document}